\RequirePackage{fix-cm}

\documentclass[smallextended]{svjour3}    
\smartqed  
\usepackage{graphicx}
\usepackage{amscd,amsfonts,amssymb,amsmath,latexsym,array,hhline,xcolor,graphicx}

\usepackage{latexsym}
\usepackage{times}

\newcommand\F{\mbox{I\kern-2pt F}}

\def\E{{\bf E}}
\def\P{{\bf P}}

\newcommand\fdem{$\Box$}
\newcommand\beq{\begin{equation}}
\newcommand\eeq{\end{equation}}
\newcommand\bea{\begin{eqnarray}}
\newcommand\eea{\end{eqnarray}}
\newcommand\bean{\begin{eqnarray*}}
\newcommand\eean{\end{eqnarray*}}

\begin{document}

\title{On the integro-differential equation arising in the ruin problem for annuity payment models}
\author{Platon Promyslov} 

\institute{\at
HSE University, Faculty of Computer Science, Moscow, 101000, Russia\\
 \email{platon.promyslov@gmail.com}
 }

\date{Received: date / Accepted: date}

\titlerunning{On the integro-differential equation}

\maketitle
 
\begin{abstract}
In the classical annuity payments model, the capital reserve of an insurance company decreases at a constant rate and increases via positive jumps. We consider a generalization of this model by supposing that a fixed portion of the capital reserve is continuously invested in a risky asset whose price follows a geometric Brownian motion, while the complementary part is placed in a bank account with a constant interest rate. The quantity of interest is the ruin probability on the infinite time horizon as a function of the initial capital. Under rather weak assumptions on the distribution of jumps, we prove that the ruin probability is the classical solution of a second-order integro-differential equation and obtain its exact asymptotics for large values of the initial capital.
\end{abstract}



\section{Introduction}
In the classical collective risk theory, initiated by Filip Lundberg in 1903 and developed further by Harald Cram\'er in the thirties, it was assumed that an insurance company kept its reserve entirely separate from risky financial activities. This was natural for the era: the legislation of the period was based on a paradigm of social responsibility and was strongly influenced by the terrible consequences of stock market disasters. However, evolving economic realities, together with progress in financial theory, led to a relaxation of legal constraints and the emergence of theoretical studies on ruin problems with risky investments. The origin of this theory can be traced back to the seminal 1993 paper by Jostein Paulsen, who applied the Kesten--Goldie implicit renewal theory (also known as the theory of distributional equations), as exposed in Goldie's \cite{Go91}, to derive asymptotics for ruin probabilities. A number of authors have since contributed to its development along the same lines, leveraging recent progress in the theory of distributional equations as summarized in the book \cite{BDM}. We mention here only a few recent papers covering Sparre Andersen type models with risky investments; see \cite{EKS}, \cite{KPro}, \cite{KLP}, and references therein. Although this method is powerful and applicable to a variety of models, it has a weakness: it provides only the rate of convergence to zero of the ruin probability, not the exact asymptotic (i.e., it gives no information on the leading constant).

For the asymptotic analysis of Lundberg--Cram\'er type models with investment, there is an alternative method based on representing the survival (or ruin) probability as a function satisfying a second-order integro-differential equation. Such an equation seems to have appeared for the first time in a short note by Anna Frolova, \cite{Fr}, who observed that in the case of exponential claims, it can be reduced to a linear ordinary differential equation (ODE). From its general solution, she conjectured that the ruin probability decays as a power function. This conjecture was later proved in \cite{FrKP} for the non-life insurance model and in \cite{KP2016} for the annuity payments model.

In the present note, we study the integro-differential equation arising in the asymptotic analysis of ruin probabilities for an annuity model more general than the one described in \cite{KP2016}. Specifically, we suppose that the insurance company pays pensions to its customers at a constant rate $c>0$ and receives, as income, the remaining part of their pension fund when the contracts expire. The company instantaneously invests a fraction $\kappa$ of its capital reserve in a risky asset whose price $S=(S_t)$ evolves as a geometric Brownian motion (gBm) with parameters $a$ and $\sigma>0$. The remaining fraction $1-\kappa$ of the capital reserve is invested in a non-risky asset (a bank account) with a rate of return $r$. Formally, the process $X=X^u$ describing the evolution of the capital reserve is given by the following stochastic differential equation with jumps:
\begin{equation}\label{one}
X_t = u + \int_0^t \kappa X_s (a\,ds +\sigma\,dW_s) + \int_0^t r(1-\kappa)X_s\,ds + P_t,
\end{equation}
where $u \ge 0$ is the initial capital, $W=(W_t)$ is a standard Wiener process and $P=(P_t)$ is the ``business'' process (or Cram\'er--Lundberg process), independent of $W$. In the classical literature, $P$ is usually represented as
\[
P_t = -ct + \sum_{i=1}^{N_t} \xi_i,
\]
where $N=(N_t)$ is a Poisson process with intensity $\lambda > 0$, and $(\xi_i)_{i\ge 1}$ is an i.i.d. sequence of strictly positive random variables, independent of $N$, with common distribution function $F$ such that $F(0)=0$. 
In the literature on stochastic calculus, it is standard to represent the sum as a stochastic integral:
\[
\sum_{i=1}^{N_t} \xi_i = \int_0^t \int_0^\infty x\,p(ds,dx) ,
\]
where $p(ds,dx)$ is a Poisson random measure with mean measure $q(ds,dx) = \lambda ds\,F(dx)$.

Such a model can also describe the capital evolution of a venture company paying current expenditures and receiving incomes $\xi_i$ from selling innovations.

After regrouping terms, the equation (\ref{one}), written in the traditional form of a stochastic differential equation, becomes:
\begin{equation}\label{two}
dX_t = \big( (a-r)\kappa + r\big) X_t\,dt + \kappa\sigma X_t\,dW_t + dP_t, \qquad X_0 = u.
\end{equation}

Let $\tau^u := \inf \{t: X^u_t \le 0\}$ (the ruin time), $\Psi(u) := \P(\tau^u < \infty)$ (the ruin probability), and $\Phi(u) := 1 - \Psi(u)$ (the survival probability).

The aforementioned paper \cite{KP2016}, together with its complements \cite{PZ} and \cite{PerErr} (treating the case $\beta=0$), deals with the asymptotic analysis of the ruin probability under the assumption of a riskless rate $r=0$. For the case where $\kappa=1$ and $F$ is an exponential distribution, it was shown that if $\beta := 2a/\sigma^2 - 1 > 0$, then $\Psi(u) \sim C u^{-\beta}$ as $u \to \infty$; if $\beta \le 0$, ruin is imminent. If $\kappa \in (0,1)$, the same conclusion holds with $\beta_\kappa := 2a/(\kappa\sigma^2) - 1$. The equation $\beta_\kappa = 0$ defines the threshold for the fraction of risky investment above which ruin is certain. The first step of the argument, which is technically difficult and delicate, is to prove that the survival probability is smooth under appropriate assumptions on $F$. The second, easier step is to prove that $\Phi$ satisfies an integro-differential equation. In the specific case of an exponential jump distribution, this equation can be reduced--by differentiating and eliminating the integral term--to a linear ODE of third order. This ODE turns out to be of second order for the derivative of the survival probability. Asymptotic results for solutions of such equations are available in the literature. The property that the ruin probability cannot decrease faster than some power function, established by Kalashnikov and Norberg in \cite{Kalash-Nor}, then yields the aforementioned result.

It is worth noting that the method of eliminating integral terms works not only for exponential jump distributions but also for other classes (e.g., Pareto distributions) and can lead to asymptotic results via Laplace transforms and Tauberian theorems; see \cite{ABL}, \cite{ACT}, and \cite{A2025}.

Importantly, if the survival probability is smooth, i.e., if $\Phi \in C^2$, then straightforward arguments based on It\^o's formula lead to the conclusion that $\Phi$ satisfies, in the classical sense, the integro-differential equation (IDE)
\begin{multline}\label{int-diff}
\frac{1}{2} \kappa^2\sigma^2 u^2 \Phi''(u) + ((a-r)\kappa + r) u \Phi'(u) =
\\
= c \Phi'(u) - \lambda \int_0^\infty (\Phi(u+y) - \Phi(u)) \,dF(y),
\end{multline}
see, e.g., \cite{KP2016}. The same equation holds for the ruin probability.

Available sufficient conditions on $F$ ensuring the smoothness of $\Phi$ are rather restrictive: the property holds if $F$ has a density $f \in C^2$ such that $f' \in L^1(\mathbb{R}_+)$; see \cite{KPukh}. For this reason, some authors prefer to assume smoothness a priori. On the other hand, it is not difficult to show that $\Phi$ is a viscosity solution of the above equation under much more general conditions, but proving uniqueness for boundary value problems in the viscosity sense is rather involved; see \cite{Bel-Kab} and, for an equation arising in a model based on a generalized Ornstein--Uhlenbeck process, the recent study \cite{AK2025}.

\smallskip
\textbf{Main result.} 
In this note, we employ a completely different proof strategy. Our main result is the following theorem.

\begin{theorem}
\label{main}
Suppose that $F(0)=0$, $\E[\xi] < \infty$ and 
$$\gamma := \dfrac{(a-r)\kappa + r}{\kappa^2 \sigma^2} > 1.$$
Then the survival probability $\Phi \in C([0,\infty)) \cap C^2((0,\infty))$ satisfies the integro-differential equation (\ref{int-diff}) with boundary conditions $\Phi(0)=0$, $\Phi(\infty)=1$. Moreover, there exists a finite constant $C > 0$ such that $\Psi(u) \sim C u^{-\gamma+1}$ as $u \to \infty$.
\end{theorem}

Our result substantially relaxes the existing conditions on the jump size distribution. Our approach is based on reducing the integro-differential equation to a pair of integral equations for the derivative of its solution. These integral equations are coupled at a suitably chosen point $u_0 \in (0,\infty)$. For the first integral equation on $[u_0,\infty)$, we establish uniqueness for a somewhat exotic terminal value problem. Specifically, we seek a solution with asymptotic behavior $u^{-\gamma}$ as $u \to \infty$, which is consistent with earlier results such as those in \cite{KP2016} (note that $\gamma=\beta+1$ in the case where $\kappa=1$ and $r=0$). It is important to note that we do not rely on previously established asymptotic formulas in our proof; they serve only as motivation. For a sufficiently large $u_0$, we find the solution using the Banach fixed point theorem. The second equation for the derivative is a Volterra integral equation of the second kind on $(0, u_0]$ with a terminal condition at $u_0$ defined via the solution of the first equation. Gluing the solutions together yields a solution for the derivative of the desired function, which includes a multiplicative constant $C$ to be determined. To satisfy the zero boundary condition at zero, the additive constant $C_0$ arising from integration to obtain the solution $\Phi$ of the original IDE must be taken as zero. The constant $C$ is then chosen to ensure that the limit at infinity equals one. Note also that $C$ is ``calculable'' in the sense that it is expressed as an integral involving the solutions of the aforementioned equations, which are themselves at least numerically computable. Finally, using It\^o's formula, we easily conclude that $\Phi$ coincides with the survival probability.

Our argument is remarkably simple. We avoid reducing the IDE to a higher-order ODE, which would introduce an additional constant, and instead rely only on standard results concerning the uniqueness of solutions to IDEs. Although conceptually inspired by the paper \cite{Grandits2004} on the non-life insurance model, our proofs are technically quite different.

\section{Analysis of the integro-differential equation} 
Recall that we assume that $F$ is continuous at zero ($F(0)=0$) and $\E \xi_1 <\infty$. 
Using integration by parts for functions of bounded variation (or Fubini's theorem), we get:
$$
\int_0^\infty \Phi(u+y) F(dy) = -\int_0^\infty \Phi(u+y) d\bar{F}(y) = \Phi(u) + \int_0^\infty \Phi'(u+y)\bar F(y)dy,
$$
where $\bar F:=1-F$. With this relation, the equation (\ref{int-diff}) can be written in the equivalent form 
\beq
\label{IDE-F}
\frac 12 \kappa^2\sigma^2u^2 \Phi''(u)+ ((a-r)\kappa +r) u\Phi'(u) -c\Phi' (u) 
=-\lambda \int_0^\infty\Phi'(u+y)\bar F(y)dy,
\eeq
suitable for our purposes. We are interested in the classical solution of this equation, that is, in $C([0,\infty])\cap C^2((0,\infty))$ with the boundary conditions at infinity and zero $\Phi(\infty)=1$ and $\Phi (0)=0$. 

The equation (\ref{IDE-F}) is reduced to a homogeneous integro-differential equation of the first order with respect to the derivative $g(u):=\Phi'(u)$. Introducing the notation 
$$
\gamma:=\frac{2((a-r)\kappa+r)}{\kappa^2 \sigma^2},\quad \alpha:=\frac{2c}{\kappa^2 \sigma^2}, \quad \mu:=\frac{2\lambda}{\kappa^2 \sigma^2}, 
 $$
we get the following IDE for $g$: 
\begin{equation}
\label{IDE}
 u^2 g'(u) + (\gamma u-\alpha) g(u)= - Ag(u),
\end{equation}
where
\beq
\label{Ag(u)}
 Ag(u):=\mu \int_0^\infty g(u+z)\,\overline{F}(z)\,dz.
\eeq

\smallskip

Note that for any $u>0$ and $p\ge 0$
$$
 |Ag(u)| \le \mu \sup_{x\ge u}x^p|g(x)| \int_0^\infty (u+z)^{-p} \bar F(z) dz \le \mu \|g\|_{p,u} \,u^{-p}\E[\xi],
$$
where $\|g\|_{p,u} :=\sup_{x\ge u}x^p|g(x)|$. The right-hand side is decreasing in $u$ and, therefore, for any $u_0>0 $ we have the bound 
\beq
\label{bound-A}
\sup_{u\ge u_0} |Ag(u)| \le \mu \|g\|_{p,u_0} \,u_0^{-p}\E[\xi]. 
\eeq

\begin{remark} If the company invests the whole capital reserve in the risky asset, i.e. $\kappa=1$, then in the notation of \cite{KP2016} the constant 
$\gamma:=2a/\sigma^2=\beta+1$. Recall that if $\beta\le 0$, then the (ultimate) ruin probability $\Psi(u)=1$ for all $u>0$. 
\end{remark}

\begin{lemma} Suppose that $\gamma>0$. Let $u_0>\mu \E[\xi]$. Then 
 the IDE (\ref{IDE}) has on $[u_0,\infty)$ a unique classical solution $g^*>0$ such that $u^\gamma g^*(u)\to 1$ as $u\to \infty$.
\end{lemma}
\begin{proof} Let $g\ge 0$ be a differentiable function satisfying (\ref{IDE}). On the arbitrary interval $[u_0,\infty)$ the equation (\ref{IDE}) can be written in the equivalent form as 
\beq
\label{IDE1}
\big( u^\gamma e^{\alpha/u} g(u) \big)' = - u^{\gamma-2}e^{\alpha/u} Ag(u).
\eeq
According to (\ref{bound-A}), $\|Ag\|_{\gamma,u_0}=\sup_{u\ge u_0} |Ag(u)|<\infty $. 
Integrating the above relation from $u$ to $\infty$ we get on $[u_0,\infty)$ the integral equation
$$
1-u^\gamma e^{\alpha/u} g(u)=- \int_u^\infty t^{\gamma-2}e^{\alpha/t} Ag(t)dt
$$
which can be rewritten in the form 
\beq
\label{g-Tg}
g(u)=g_0(u)+Tg(u)
\eeq
where $g_0(u):=u^{-\gamma} e^{-\alpha/u}$ and 
$$
 T g(u) = \mu\,u^{-\gamma} e^{-\alpha/u} \int_{u}^{\infty} t^{\gamma-2}\,e^{\alpha/t} \left(\!\int_{0}^{\infty} g(t+z)\bar F(z)dz\right)\,dt.
$$

Let us consider the Banach space 
$$
\mathcal{X}_{\gamma,u_0} = \left\{ g \in C([u_0, \infty]) \colon \|g\|_{\gamma,u_0} < \infty \right\}. 
$$
Note that for $u\ge u_0$ and any $g\in \mathcal{X}_{\gamma,u_0}$
$$
u^{\gamma} T g(u) \le \mu \int_{u_0}^{\infty} t^{\gamma-2}\,\left(\!\int_{0}^{\infty} g(t+z)\bar F(z)dz\right)\,dt\le \E[\xi]
\mu\,u_0^{-1}\|g\|_{\gamma,u_0} <\infty. 
$$
It follows that the linear operator $T$ maps the space $\mathcal{X}_{\gamma,u_0}$ into itself. 
Moreover, for any $g_1,g_2\in \mathcal{X}_{\gamma,u_0}$ we have the bound 
$$
||T g_1-T g_2||_{\gamma,u_0}\le 
\theta \|g_1-g_2\|_{\gamma,u_0}, \qquad \theta:=\mu \E[\xi]u_0^{-1}.
$$
Due to our choice of $u_0$ the constant $\theta<1$, so that the mapping $T$ is a contraction of the Banach space $\mathcal{X}_{\gamma,u_0}$. Since $g_0\in \mathcal{X}_{\gamma,u_0}$, the mapping $S: g\mapsto g_0+Tg$ also is a contraction of $\mathcal{X}_{\gamma,u_0}$. 
By the Banach theorem $S$ has a unique fixed point $g^*$. Since $g_0>0$ the successive iterations are increasing, the fixed point $g^*$ is a strictly positive function on $[u_0,\infty)$. It is easy to check that the function $S (g^*)$ is continuously differentiable and so is $g^*$. \fdem 
\end{proof}

\begin{remark} 
\label{rem1}
The solution $g^*$ on $[u_0,\infty)$ is obtained for a fixed $u_0$, hence, depends on the initial point $u_0$. In symbols, it can be written as $g(u;u_0)$. It is easily seen that this function remains the solution also on the interval $[u_1,\infty)$ for $u_1\ge u_0$. Due to the uniqueness we have that $g^*(u;u_0)=g^*(u;u_1)$ on $[u_1,\infty)$. 
\end{remark}

\bigskip
\begin{lemma} Suppose that $\gamma>1$. Then the IDE (\ref{IDE}) has on $[0,\infty)$ a unique classical solution $\hat g>0$ such that $\hat g=g^*$ on $[u_0,\infty)$ (therefore, $u^\gamma \hat g(u)\to 1$ as $u\to \infty$).
\end{lemma}
\begin{proof} To solve the IDE (\ref{IDE1}) (equivalent to (\ref{IDE})) on the interval $(0,u_0)$ with the terminal condition 
$v(u_0)=g^*(u_0)$ we first transform it to the integral equation 
\beq
\label{IE}
u^\gamma e^{\alpha/u} g(u)=u_0^\gamma e^{\alpha/u_0} g^*(u_0)+\int_u^{u_0}y^{\gamma-2}e^{\alpha/y} Ag(y)dy.
\eeq
Splitting the integral in the definition of $Ag(u)$ at the point $u_0$ we get after change of variables that 
$$
 Ag(u)
 =\mu \int_u^{u_0} g(y)\overline{F}(y-u)\,dy + H(u), 
$$
where the function 
$$
 H(u):= \mu \int_{u_0}^\infty g^*(y)\overline{F}(y-u)\,dy
$$
is already known. It is increasing and bounded: 
$$
H(u)\le \mu \sup_{y\ge u_0}g^*(y)\int_u^\infty\bar F(u-y)dy= \mu \sup_{y\ge u_0}g^*(y) \E\xi. 
$$
Using the dominated convergence, we get that the limit $H(0+)$ exists and equal to 
\beq
\label{H0}
H(0)=\mu \int_{u_0}^\infty g^*(y)\overline{F}(y-)\,dy>0. 
\eeq

Let $M(u):=u^\gamma e^{\alpha/u}$. Dividing by $M(u)$ and changing the order of integration, we get from (\ref{IE}) 
 the equation 
\begin{equation}
\label{eq:volterra_final+}
g(u) = f(u) + \int_{u}^{u_{0}} K(u, y)g(y)\,dy, \quad u \in (0, u_0],
\end{equation}
where 
$$
f(u):=\frac {M(u_0)}{M(u)}g^*(u_0)+\frac {1}{M(u)}\int_u^{u_0} t^{\gamma-2}e^{\alpha/t} H(t)\,dt, 
$$
$$
K(u, y):=\mu\frac{1}{M(u)} \int_u^y t^{\gamma-2}e^{\alpha/t} \overline{F}(y-t)\,dt.
$$

By Theorem 2.1.1 from the book \cite{Burton} the relation \eqref{eq:volterra_final+} is a Volterra equation of the second kind with continuous kernel. It admits a unique solution $g_* \in C((0, u_0])$. It is also clear that $g_*\in C^1((0, u_0))$. 

Note that $f$ is continuous function on $(0,u_0]$. Let us 
check that it can be extended as a continuous function on the closed interval $(0,u_0]$.
Indeed, $M'(u)=(\gamma u -\alpha )u^{\gamma-2}e^{\alpha/ u}$ and by the L'H\^{o}pital rule
$$
\lim_{u\downarrow 0} f(u)=\lim_{u\downarrow 0}\frac {1}{M(u)}\int_u^{u_0} t^{\gamma-2}e^{\alpha/t} H(t)\,dt=
-\lim_{u\downarrow 0}\frac{1}{M'(u)}u^{\gamma-2}e^{\alpha/u} H(u)=\frac{H(0)}{\alpha},
$$
where $H(0)$ is given by the formula (\ref{H0}). 
It follows that $f$ can be extended as a continuous function on the closed interval $[0,u_0]$. Hence, it is bounded from above by some constant $\kappa>0$. 

A similar calculation shows that for all $y\in (0,u_0]$
$$
\lim_{u\downarrow 0} K(u,y)=\frac {\mu}{\alpha}\bar F(y-). 
$$

For all $u\in (0,u_0$) and $y\ge u$ we have that 
$$
 K(u,y)\le \mu \frac{1}{M(u)} \int_u^{y}e^{\alpha/t} t^{\gamma-2}\,dt\le \frac {\mu}{\gamma -1}y^{\gamma-1}:=h(y)
$$

Taking into account these properties we get that for the solution $g_*$ of (\ref{eq:volterra_final+}) the bound 
$$
g_*(u)\le \kappa+ \int_u^{u_0}h(y)g_*(y)dy, \qquad u \in (0,u_0].
$$
By the Gronwall--Bellman inequality 
$$
g_*(u)\le \kappa \exp \left\{ \int_u^{u_0}h(y)\,dy \right\}.
$$
Hence, $g_*(u)$ is bounded on $(0,u_0]$. 

Passing to the limit in (\ref{eq:volterra_final+}) as $u\downarrow 0$ we get that 
$$
g_*(0+)=\frac{H(0)}{\alpha}+\frac{\mu }{\alpha}\int_0^{u_0}g_*(y)\bar F(y) \,dy=\frac {\mu}{\alpha}\int_{u_0}^{\infty}g^*(y)\bar F(y)\,dy+\frac{\mu }{\alpha}\int_0^{u_0}g_*(y)\bar F(y) \,dy
$$
and we can extend $g_*$ as a continuous function on the closed interval $[0,\infty]$ by putting $g_*(0):=g_*(0+)$. 
It follows that the function $\hat g:= g_*I_{[0,u_0]}+g^*I_{(u_0,\infty)}$ is the unique classical solution of the equation (\ref{IDE}) such that 
$u^\gamma \hat g(u)\to \infty$ as $u\to \infty$. 

Note that there is no problem with derivative at the point $u_0$. The constructed solution due to the announced uniqueness can be replaced by an arbitrary point $u_1>u_0$ without influencing the solution $\hat g$. \fdem
\end{proof}

Thus, $\hat g$ is the classical solution. 

\section{The probability of non-ruin and the solution of IDE} 
Due to linearity of the equation (\ref{IDE}) all its solutions with the asymptotic behavior at infinity $C u^{-\gamma}$, $C>0$, are of the form $C\hat g(u)$. The corresponding general solution of the equation (\ref{int-diff}) has the form 
\beq
\label{GC0C}
G(u)=C_0+C\int_0^u \hat g(t)\,dt. 
\eeq
Our aim is to choose the constants to ensure that $G=\Phi$. The following lemma shows that the constant $C_0$ must be equal to zero.
\begin{lemma}
$\Phi(0+)=\Phi(0):=0$.
\end{lemma}
\begin{proof} 
Consider the stochastic exponential $Z=(Z_t)_{t\ge 0}$ given by the linear SDE 
$$
dZ_t=Z_t \left( ((a-r)\kappa + r) \,d t + \sigma \,dW_t \right), \qquad Z_0=1, 
$$
and define the process $\hat X^u$ 
$$
\hat X^u_t=Z_t\Big( u-c\int_0^t Z_s^{-1}ds\Big).
$$

Let $\hat \tau^ u:=\inf\{t\ge 0\colon \hat X^u=0\}$ and let $T_1$ be the instant of first jump of the Poisson process $N$. 
According to the stochastic Cauchy formula giving the solution of non-homogeneous linear SDE 
$X^u=\hat X^u$ on $[0,T_1)$ and $X^u_{T_1}=\hat X^u_{T_1}+\Delta X^u_{T_1}\ge \hat X^u_{T_1}$. Thus, $\hat \tau^ u\wedge T_1= \tau^ u\wedge T_1$. It follows that 
$$
0\le \hat X^u_{\tau^u\wedge T_1}=Z_{\tau^u\wedge T_1}\Big( u-c\int_0^{\tau^u\wedge T_1} Z_s^{-1}ds\Big)
\le \xi^* u-c\frac {\xi_*}{\xi^*}(\tau^u\wedge T_1),
$$
where $\xi_*:=\inf_{s\le T_1}Z_s$, $\xi^*:=\sup_{s\le T_1}Z_s$. Therefore, $\tau^ u\to 0$ as $u\to 0$ and the result follows.
\fdem
\end{proof} 

We choose the constant $C$ to have the equality $G(\infty)=\Phi(\infty) = 1 $.
 
\begin{lemma}
The function $G$ given by (\ref{GC0C}) with $C_0=0$ and $C=1/\int_0^\infty \hat g(s)\, ds$ coincides with $\Phi$. 
\end{lemma}
\begin{proof} Let us apply It\^o's formula to $(X^u_{t\wedge \tau})$ where $G$ which is the classical solution of (\ref{int-diff}) with $G(0)=0$ and $G(\infty)=1$. Compensating the jump component and regrouping terms we observe that the integrand of the integral with respect to $ds$ is equal to zero in virtue of (\ref{int-diff}) and, therefore, $G(X^u_{t\wedge \tau^u})-G(u)=M_t$, where 
$$
M_t=\int_0^{t\wedge \tau^u}G'(X^u_{s\wedge \tau^u)}) \kappa\sigma X^u_s dW_s+\int_0^{t\wedge \tau^u}\int (G(X^u_{s-}+x)-G(X^u_{s-})(p-q)(ds,dx). 
$$
The process $(M_t)$ is a bounded martingale starting from zero. Hence, $\E [M_\tau]=0$ for any stopping time $\tau$. 
It follows that 
$$
G(u)
= \E\left[ G(X^u_{\tau^u})I_{\{\tau^u=\infty\}}\right]+\E\left[ G(X^u_{\tau^u})I_{\{\tau^u<\infty\}}\right]=G(\infty)\P(\tau^u=\infty)=\Phi(u) 
$$ 
and we get the result. \fdem
\end{proof}

Since $\Phi=G$ and, by the L'h\^opital rule $1-G(u)\sim Cu^{-\gamma + 1}$, as $u\to \infty$ we get all properties announced in Theorem \ref{main}.

\section*{Competing interests}
The author declares no competing interests.

\acknowledgement 
The author expresses his gratitude to Yuri Kabanov for the problem formulation and helpful guidance.

\end{document}